\documentclass[12pt]{article}

\DeclareMathAlphabet{\mathpzc}{OT1}{pzc}{m}{it}

\usepackage{amsmath}
\usepackage{amssymb}
\usepackage{amsfonts}
\usepackage{amsthm}
\usepackage{mathrsfs}
\usepackage{enumerate}
\usepackage[pdftex]{color, graphicx}

\newcommand{\cC}{{\mathcal C}}

\newcommand{\cS}{{\mathcal S}}

\newcommand{\pa}{\parallel}
\newcommand{\cd}{\cdot}

\newcommand{\Vi}{V_\infty}
\newcommand{\li}{l_\infty}
\newcommand{\la}{\lambda}

\allowdisplaybreaks

\newtheorem{thm}{Theorem}[section]
\newtheorem{lem}[thm]{Lemma}

\newtheorem{prop}[thm]{Proposition}
\newtheorem{cor}[thm]{Corollary}

\begin{document}

\renewcommand{\thefootnote}{\arabic{footnote}}
 	
\title{Synthetic foundations of cevian geometry, II:\\The center of the cevian conic}

\author{\renewcommand{\thefootnote}{\arabic{footnote}}
Igor Minevich and Patrick Morton\footnotemark[1]}
\footnotetext[1]{The authors were supported by an Honors Research Fellowship from the IUPUI Honors Program from 2007-2009, during the period in which this article was written.  The results of this paper were part of an undergraduate research project.}
\maketitle

\begin{section}{Introduction.}
\end{section}

In this paper we continue the investigation begun in Part I \cite{mm1} , and study the conic $\mathcal{C}_P=ABCPQ$ on the five points $A, B, C, P, Q$, where $Q=K \circ \iota(P)$ is the isotomcomplement of $P$, defined to be the complement of the isotomic conjugate of $P$ with respect to $ABC$.  Here, as in Part I, $P$ is any point not on the extended sides of $ABC$ or its anticomplementary triangle.  This conic is defined whenever $P$ does not lie on a median of triangle $ABC$.  We show first that the symmetrically defined points $P'=\iota(P)$ and $Q'= K \circ \iota(P')=K(P)$ also lie on this conic, as well as 6 other points that can be given explicitly (see Theorem 2.1 and Figure 1).  \medskip

Recall from Part I that the map $T_P$ is the unique affine map which takes triangle $ABC$ to the cevian triangle $DEF$ of $P$ with respect to $ABC$.  The map $T_{P'}$ is defined in the same way for the point $P'$.  In Theorem 2.4 we show that if $P$ and $P'$ are ordinary points and do not lie on a median of $ABC$, then
\begin{equation}
\eta T_P=T_{P'} \eta,
\end{equation}
where $\eta$ is the harmonic homology (affine reflection, see \cite{cox3}) whose center is the infinite point $V_\infty$ on the line $PP'$ and whose line of fixed points is the line $GV$, where $G$ is the centroid of $ABC$ and $V=PQ \cdot P'Q'$.  Thus, $T_P$ and $T_{P'}$ are conjugate maps in the affine group.  We prove this formula synthetically by proving an interesting relationship between the centroids $G_1$ and $G_2$ of the cevian triangles $DEF$ and $D_3E_3F_3$ of the points $P$ and $P'$, respectively.  Lemma 2.5 shows that $G$ is the midpoint of the segment $G_1G_2$ and $\eta(G_1)=G_2$.  \medskip

After we introduce the affine map
$$\lambda= T_{P'} \circ T_P^{-1}$$
in Section 3, where $T_P$ and $T_{P'}$ are the affine maps considered in Part I, we show that the 6 points on $\cC_P$ mentioned above are the images of the vertices $A, B, C$ under $\lambda$ and $\lambda^{-1}$ (Theorem 3.4).  The mapping $\lambda$ leaves the conic $\mathcal{C}_P$ invariant as a set (Theorem 3.2), and is the main affine map considered in this paper.  The relation (1) allows us to write the map $\lambda$ as
$$\lambda=\eta \circ (T_P \circ \eta \circ T_P^{-1})=\eta_1 \circ \eta_2,$$
where both maps $\eta_1$ and $\eta_2$ are harmonic homologies.  Using this representation we prove in Theorem 4.1 that the center $Z=Z_P$ of the conic $\mathcal{C}_P$ is the intersection
$$Z=GV \cdot T_P(GV),$$
and that when $\mathcal{C}_P$ is a parabola or an ellipse, $Z$ is the unique fixed point of $\lambda$ in the extended plane.  When $\mathcal{C}_P$ is a hyperbola, $Z$ is the unique ordinary fixed point of $\lambda$.  In the latter case, $\lambda$ also fixes the two points at infinity on the asymptotes of $\mathcal{C}_P$.  \medskip

At the end of the paper we interpret the mapping $\lambda$ as an isometry on the model of hyperbolic geometry whose points are the interior points of the conic $\mathcal{C}_P$ and whose lines are Euclidean chords.  \medskip

In Part III \cite{mm3} of this series of papers we will prove that the point $Z$ is a generalized Feuerbach point.  In particular, this will show that $Z$ is the Feuerbach point of triangle $ABC$ when $P$ is the Gergonne point and $P'$ is the Nagel point of $ABC$ (see [ac] and [ki]).  Our Theorem 4.1 therefore gives a representation of the Feuerbach point as the intersection of two lines.  Corollary 4.2 shows that these two lines are $GV$ and $T_P(GV)=G_1J$, where $J$ is the midpoint of $PQ$.  A third line through $Z$ is the line $G_2J'$, where $G_2=T_{P'}(G)$ and $J'$ is the midpoint of $P'Q'$.  (See Figure 5 in Section 4.)  In this case $\mathcal{C}_P$ is a hyperbola, and the Feuerbach point $Z$ is the unique ordinary fixed point of $\lambda=T_{P'} \circ T_P^{-1}$.  \medskip

We mention one more fact that we prove along the way.  The mapping
$$\mathcal{S}'= T_{P'}T_PT_{P'}^{-1}T_P^{-1}$$
is always a translation, which implies (Corollary 2.8 and Figure 4) that the triangles $A_3B_3C_3=T_P(D_3E_3F_3)$ and $A_3' B_3' C_3'=T_{P'}(DEF)$ are always congruent triangles. \medskip

We refer to Part I \cite{mm1} for the notation that we use throughout this series of papers; to [ac], [wo], [y1], or [y2] for definitions in triangle geometry; and to [cox1] and [cox2] for classical results from projective geometry.

\bigskip

\bigskip

\begin{section}{The conic $ABCPQ$ and the affine mapping $T_P$.}

We start by proving

\begin{thm}\label{thm:3.1} If $P$ does not lie on the sides of triangle $ABC$ or its anticomplementary triangle, and not on a median of $ABC$, then there is a conic $\cC_P$ on the points $A, B, C, P, Q, P', Q'$. This conic also passes through the points
\begin{equation}
\label{eqn:3.1}
A_0P\cd D_0Q', \quad B_0P\cd E_0Q', \quad C_0P\cd F_0Q',
\end{equation}
and
\[A'_0P'\cd D_0Q, \quad B'_0P'\cd E_0Q, \quad C'_0P'\cd F_0Q.\]
\end{thm}

\begin{proof}
The condition that $P$ is not on a median of $ABC$ ensures that the points $P$ and $Q$ are not collinear with one of the vertices. If $P, Q$, and $A$ are collinear, for example, then the points $D$ and $D_2$ coincide, so $A_1$ and $A_2$ coincide, meaning that $Q'$ is collinear with $A$ and $G$ (by I, Theorem 3.5). But $P$ is collinear with $K(P) = Q'$ and $G$, so $P$ would lie on $AG$. By the same reasoning, $P$, $Q'$, and $A$ are not collinear and neither are $P, P'$ and $A$.\\

Now the mapping $T_P$ is a projective mapping which takes the pencil of lines $x$ on $Q'$ to the pencil $y = T_P(x)$ on $P$, since $T_P(Q') = P$. For the lines $x = AQ'$ and $y = DP = AP$ we have $x\cd y = A$; while $x = BQ'$ and $y = EP = BP$ give $x\cd y = B$; and $x = CQ'$ and $y = FP = CP$ give $x\cd y = C$. Thus the pencil $x$ is not perspective to the pencil $y$, so Steiner's theorem [cox2, p. 80] implies that the locus of points $x \cd y$ is the conic $ABCPQ'$. If $x = QQ'$ then $y = QP$, so $x \cd y = Q$ is also on this conic. Hence the conic $ABCPQ' = ABCQQ'$. Arguing the same with the mapping $T_{P'}$ shows that there is a conic $ABCP'Q = ABCQQ'$. It follows that the conic $\cC_P = ABCPQ'$ lies on $P'$ and $Q$, so that $\cC_P = ABCPQ$. This proves the first assertion. Letting $x = D_0Q'$ gives $y = A_0P$, so $A_0P\cd D_0Q'$ is on $\cC_P$, as are all the other listed intersections.
\end{proof}

\begin{cor}\label{cor:3.1}
\begin{enumerate}[(a)]
\item If $Y$ is any point on the conic $\cC_P$ other than $P, Q'$ (respectively $Q, P'$), then $T_P(Q'Y) = PY$ (resp. $T_{P'}(QY) = P'Y$).
\item The conic $\mathcal{C}_P$ is the locus of points $Y$ for which $P, Y$, and $T_P(Y)$ are collinear.
\item In particular, $P, P', T_P(P')$, and $T_{P'}(P)$ are collinear (whether $P$ lies on a median or not).
\end{enumerate}
\end{cor}

\begin{proof}
Part (a) follows from the definition of the conic $\cC_P$ as the locus of intersections $x\cd y$. For part (b), if $Y$ lies on $\cC_P$, then by part (a), $P, Y$, and $T_P(Y)$ are collinear. This also clearly holds for $Y=P$ and $Y=Q'$.  Conversely, suppose $P, Y$, and $T_P(Y)$ are collinear. If $Y \notin \{P,Q'\}$, then $T_P$ maps the line $x = Q'Y$ to the line $y = PT_P(Y) = PY$, so $x\cd y = Y$ lies on $\cC_P$. Part (c) follows from part (b) with $Y = P'$ and from the analogous statement obtained by switching $P$ and $P'$, as long as $P$ does not lie on a median of $ABC$. Now assume that $P\ne G$ lies on the median $AG$. Then $P, P', Q$, and $Q'$ are all on $AG$. From the Collinearity Theorem (I, Theorem 3.5) the points $D_i$ for $0 \le i \le 4$ are all the same point, so $A_0 = A_2 = A_3$, the midpoint $EF\cd AG$ of $EF$; and similarly, the points $A'_i$ equal $A'_0 = A'_3$, the midpoint $E_3F_3\cd AG$ of $E_3F_3$. From $P'$ on $AD_3$ it follows that $T_P(P')$ is on $DA_3=D_0A_2 = AG$. Similarly, $T_{P'}(P)$ lies on $AG = PP'$. 
\end{proof}

\begin{figure}
\[\includegraphics[width=4.5in]{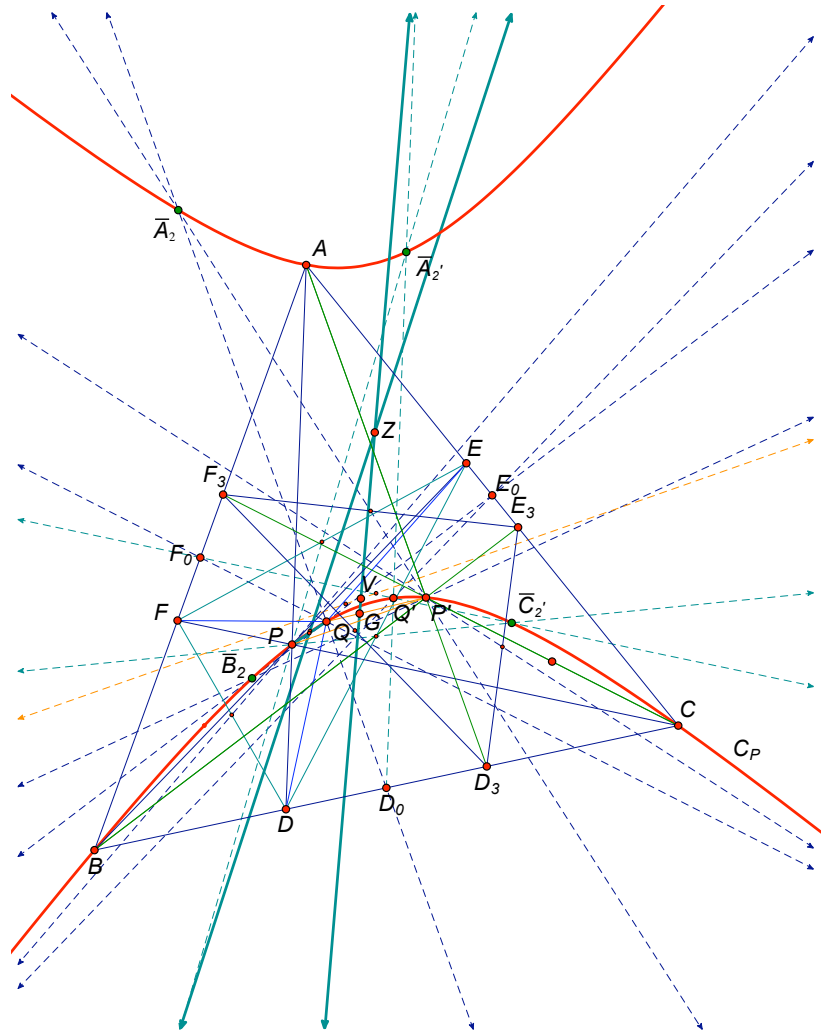}\]
\caption{Cevian conic, with points $\bar A_2=A_0'P' \cdot D_0Q, \bar A_2'=A_0P \cdot D_0Q'$, etc.}
\label{fig:3.0}
\end{figure}
\bigskip

\noindent {\bf Remarks}. 1. We know from I, Theorems 3.13 and 2.4 that the lines $D_0Q'$, $E_0Q'$, $F_0Q'$ in (1) are the cevians for the point $Q'$ with respect to the anticevian triangle $A'B'C'$ of $Q$, since $A'Q', B'Q'$, and $C'Q'$ pass through the midpoints of the sides of $ABC$. See Theorem \ref{thm:3.9} below.\\
2. Part (b) of the corollary allows for an easy computation of an equation for the conic $\cC_P$. \medskip

We will assume the hypothesis of Theorem \ref{thm:3.1} anytime we make use of the conic $\cC_P = ABCPQ$. Alternatively, we could define $\cC_P$ to be a degenerate conic, the union of a median and a side of $ABC$, when $P$ does lie on a median. In that case, $P, Q, P', Q'$ are all on the median.\medskip

Assume now that the point $P$ is ordinary and does not lie on a median of triangle $ABC$ or on $\iota(\li)$. Then the points $P'$ and $Q$ are also ordinary. We shall use the conic $\cC_P$ to prove an interesting relationship between the maps $T_P$ and $T_{P'}$. First note that the lines $PP'$ and $QQ'$ are parallel, since $K(PGP') = Q'GQ$. The midpoint of $QQ'$ is clearly the complement of the midpoint of $PP'$, so the line joining them passes through the centroid $G$. Since the quadrangle $PP'QQ'$ is inscribed in $\cC_P$, its diagonal triangle is a self-polar triangle for $\cC_P$ [cox2, p. 75]. The vertices of this self-polar triangle are $PQ'\cd P'Q = G, PQ\cd P'Q' = V$, and $PP'\cd QQ' = V_\infty$, a point on the line at infinity. Hence, the polar of $G$ is the line $VV_\infty$, which is the line through $V$ parallel to $PP'$. The polar of $V_\infty$ is $GV$. Since the segment $QQ'$ is parallel to $PP'$ and half its length, $Q$ is the midpoint of segment $PV$, so $V$ is the reflection of $P$ in $Q$, and the reflection of $P'$ in $Q'$.  Considering the quadrangle $VQGQ'$, it is not hard to see that $GV$ passes through the midpoints of segments $PP'$ and $QQ'$, since these midpoints are harmonic conjugates of $V_\infty$ with respect to the point pairs $(QQ')$ and $(PP')$.  \medskip

Let $Z = Z_P$ be the center of the conic $\cC_P = ABCPQ$, so $Z$ is the pole of the line at infinity. Since $V_\infty$ lies on the polar of $Z$, $Z$ must lie on $GV$. Since $PP'$ and $QQ'$ are parallel chords on the conic $\cC_P$, the line through their midpoints passes through $Z$ [cox1, p. 111]. Hence we have:

\begin{prop}\label{prop:3.2} Assume that the ordinary point $P$ does not lie on a median of triangle $ABC$ or on $\iota(\li)$.
\begin{enumerate}[(a)]
\item The points $G, V = PQ\cd P'Q'$, and $V_\infty = PP'\cd QQ'$ form a self-polar triangle with respect to the conic $\cC_P = ABCPQ$.
\item The center $Z$ of the conic $\cC_P$ lies on the line $GV = G(PQ\cd P'Q')$, which is the polar of the point $V_\infty$.
\item The line $GV = GZ$ passes through the midpoints of the parallel chords $PP'$ and $QQ'$.
\item The harmonic homology $\mu_G$ with center $G$ and axis $VV_\infty$ maps the conic $\cC_P$ to itself. In other words, if a line $GY$ intersects the conic in points $X_1$ and $X_2$ and $GY\cd VV_\infty = X_3$, then the cross-ratio $(X_1X_2, GX_3) = -1$.
\item The line $VV_\infty$ is the same as the line joining the anti-complements of $P$ and $P'$. Thus, the polar of $G$ is the line $K^{-1}(PP')$.
\item The point $V$ is the midpoint of the segment joining $K^{-1}(P)$ and $K^{-1}(P')$.
\end{enumerate}
\end{prop}

\begin{proof}
We have already proven parts (a)-(c). Part (d) is immediate from the fact that a conic is mapped into itself by any homology whose center is the pole of its axis [cox2, p. 76, ex. 4]. For part (e), note that $\mu_G(P) = Q'$, since $PQ'$ is a chord of the conic containing $G$. Hence $(PQ', GX_3) = -1$, where $GP\cd VV_\infty = X_3$; this implies $PX_3 = -2X_3Q'$, which means that $Q'$ is the midpoint of $PX_3$. But $Q' = K(P)$, so $K(X_3) = P$. Thus $X_3$ is the anti-complement of $P$. Similarly, the anti-complement of $P'$ lies on $VV_\infty$, and this proves part (e). Part (f) follows from the fact that the midpoint $M$ of $PP'$ lies on the line $GV$, so that
\[K^{-1}(M) = K^{-1}(PP')\cd K^{-1}(GV) = K^{-1}(PP')\cd GV = V,\]
since the line $GV$ is an invariant line of the complement map.
\end{proof}

Let $\eta = \eta_P$ be the harmonic homology whose center is $V_\infty$ and whose axis is its polar $GV$. The map $\eta$ is an affine reflection [cox3, p. 203], since it fixes the line $GV$ and maps a point $Y$ to the point $Y'$ with the property that $YY'$ is parallel to $VV_\infty$ (or $PP'$) and $YY'\cd GV$ is the midpoint of $YY'$. The map $\eta$ is an involution on the extended plane. It takes the conic $\cC_P$ to itself and interchanges the point pairs $(PP')$ and $(QQ')$, since the line $GV$ passes through the midpoints of the chords $PP'$ and $QQ'$ and both lines $PP'$ and $QQ'$ lie on $V_\infty$. Hence this homology induces an involution of points on $\cC_P$.\\
\\
{\bf Remark.} It is not hard to show that the map $\eta$ commutes with the complement map: $K\eta = \eta K$.\medskip

We shall now prove

\begin{thm}\label{thm:3.3} Assume that the ordinary point $P$ does not lie on a median of triangle $ABC$ or on $\iota(\li)$. Then the maps $T_P$ and $T_{P'}$ satisfy the equation $\eta T_P = T_{P'} \eta$, and so are conjugate to each other in the affine group.
\end{thm}

To prove this theorem we need a lemma, which is of interest in its own right.

\begin{lem}\label{lem:3.4}
Let $G_1 = T_P(G)$ and $G_2 = T_{P'}(G)$ be the centroids of the cevian triangles $DEF$ and $D_3E_3F_3$ of $P$ and $P'$. Then the centroid $G$ of $ABC$ is the midpoint of the segment $G_1G_2$, which is parallel to $PP'$. In other words, $\eta(G_1) = G_2$ (when $P$ and $P'$ are ordinary).
\end{lem}

\begin{proof}
We first show that the line $G_1G_2$ lies on $G$ and is parallel to $PP'$. To begin with, assume $P$ and $P'$ are ordinary. Using I, Corollary 3.3, we know that the points $Q, G_1, T_P(P')$ are collinear, with $G_1$ one-third of the way from $Q$ to $T_P(P')$. Since $G$ is collinear with $Q$ and $P'$ and one-third of the way along $QP'$, the triangles $G_1QG$ and $T_P(P')QP'$ are similar (SAS). Hence the line $G_1G$ is parallel to the line $T_P(P')P' = PP'$ (Corollary \ref{cor:3.1}). Switching the roles of $P$ and $P'$ gives that $G_2G$ is also parallel to $PP'$, which implies that $G_1G = G_2G$, proving the claim.\\

If $P' = Q$ is infinite, then $P$ and $Q'$ are ordinary, so we still get $G_2G \pa PP'$. Applying the map $T_P$ and using I, Theorem 3.14 gives that $T_P(G_2G) \pa T_P(PP')$, i.e., $GG_1 \pa PP'$, because $T_P(G_2)=T_P T_{P'}(G)=K^{-1}(G)=G$ and $T_P(PP') = T_P(PQ) = T_P(P)Q \pa PQ = PP'$. Hence we get $G_1G = G_2G$, as before.  A similar argument works if $P = Q'$ is infinite.\\

We now show that $G$ is the midpoint of $G_1G_2$. (Cf. Figure 3.)  Consider the sequence of triangles $DEF$, $D_0EF, D_3EF$ with centroids $G_1, G_{01}, G_{11}$ lying on the respective lines $A_0D, A_0D_0, A_0D_3$ (since $A_0$ is the midpoint of $EF$). By the properties of the centroid we know that the segment $G_1G_{11}$ is the image of the segment $DD_3$ under a dilatation with center $A_0$ and ratio 1/3. Since $D_0$ is the midpoint of $DD_3$ it follows that the vector $\overrightarrow{G_1G_{01}}$ is $\frac12$ the vector $\overrightarrow{G_1G_{11}}$.

\begin{figure}
\[\includegraphics[width=4.5in]{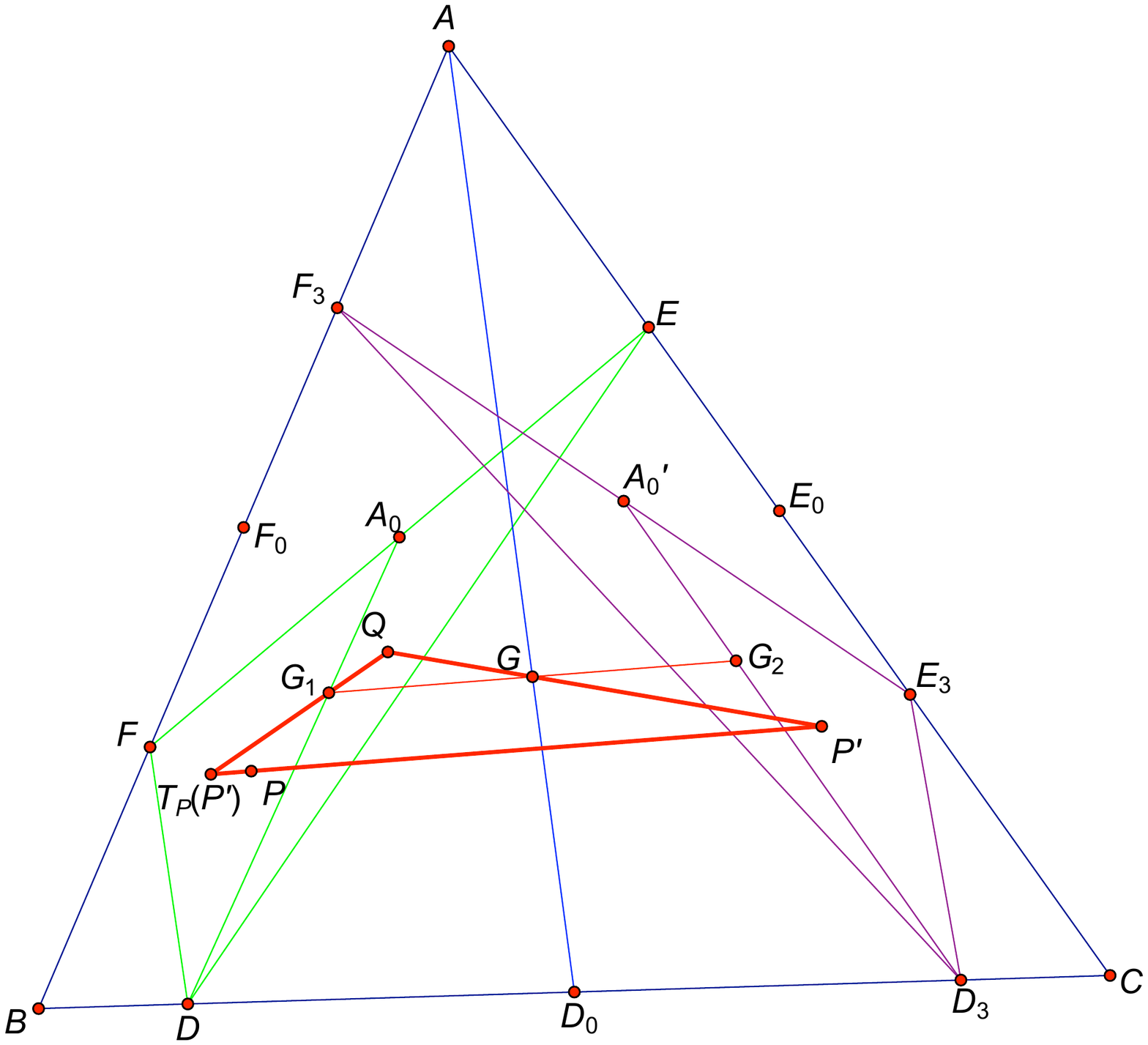}\]
\caption{$G_1G_2 || PP'$}
\label{fig:3.1}
\end{figure}
\bigskip

Next consider passing from triangle $D_3EF$ with centroid $G_{11}$ to the triangle $D_3E_3F$ with centroid $G_{12}$. Considering the dilatation with ratio $1/3$ from the midpoint of $D_3F$ shows that the vector $\overrightarrow{G_{11}G_{12}}$ is $\frac13$ the vector $\overrightarrow{EE_3}$. In the same way, if we move from triangle $D_0EF$ with centroid $G_{01}$ to triangle $D_0E_0F$ with centroid $G_{02}$, the vector $\overrightarrow{G_{01}G_{02}}$ is $\frac13$ the vector $\overrightarrow{EE_0}$, and therefore $\frac12$ the vector $\overrightarrow{G_{11}G_{12}}$.\\

Finally, in passing from triangle $D_3E_3F$ with centroid $G_{12}$ to triangle $D_3E_3F_3$ with centroid $G_2$, the vector $\overrightarrow{G_{12}G_2}$ is $\frac13$ the vector $\overrightarrow{FF_3}$. Similarly, in passing from $D_0E_0F$ with centroid $G_{02}$ to triangle $D_0E_0F_0$ with centroid $G$, the vector $\overrightarrow{G_{02}G}$ is $\frac13$ of $\overrightarrow{FF_0}$. Hence, the vector $\overrightarrow{G_{02}G}$ is $\frac12$ the vector $\overrightarrow{G_{12}G_2}$.\\

It follows that in passing from triangle $DEF$ to $D_0E_0F_0$, the centroid experiences a displacement represented by the vector
\[\overrightarrow{G_1G_{01}} + \overrightarrow{G_{01}G_{02}} + \overrightarrow{G_{02}G} = \frac12(\overrightarrow{G_1G_{11}} + \overrightarrow{G_{11}G_{12}} + \overrightarrow{G_{12}G_2}),\]
which is $\frac12$ the vector displacement from $G_1$ to $G_2$. Hence, $G_1G = \frac12 G_1G_2$, which proves that $G$ is the midpoint of $G_1G_2$.
\end{proof}

\begin{figure}
\[\includegraphics[width=4.5in]{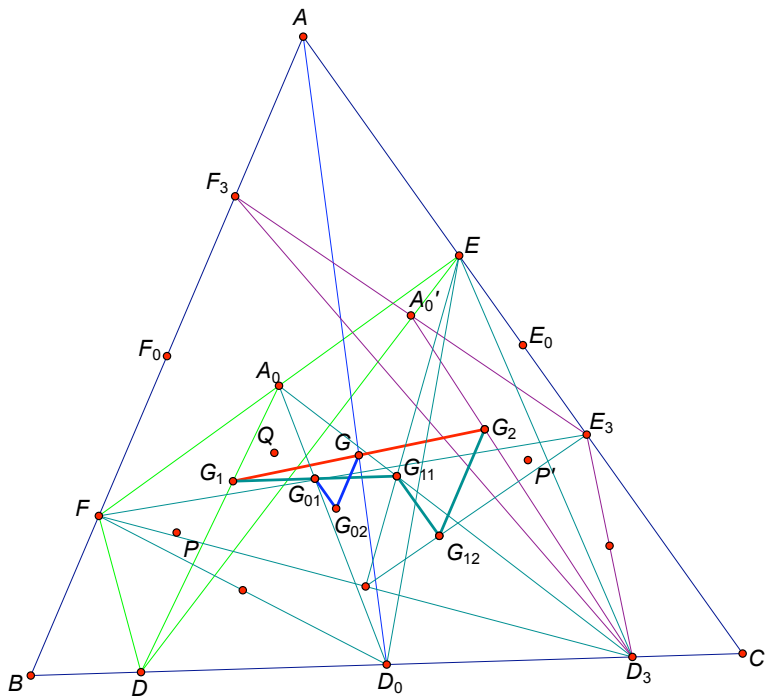}\]
\caption{$G$ is the midpoint of $G_1G_2$}
\label{fig:3.2}
\end{figure}

\begin{proof}[{Proof of Theorem 2.4}]
We check that $\eta T_P(Y) = T_{P'}\eta(Y)$ for three non-collinear ordinary points $Y$. This holds for $Y = G$ by the lemma. It also holds for the points $Q$ and $Q'$, since I, Theorems 3.2 and 3.7 imply that
\[\eta T_P(Q) = \eta(Q) = Q' = T_{P'}(Q') = T_{P'}\eta(Q)\]
and
\[\eta T_P(Q') = \eta(P) = P' = T_{P'}(Q) = T_{P'}\eta(Q').\]
The points $Q, Q'$, and $G$ are clearly not collinear, since $G$ is the intersection of the diagonals in the trapezoid $PP'Q'Q$ (and no three of these points lie on a line because they all lie on the conic $\mathcal{C}_P$). This implies the theorem, since $\eta T_P$ and $T_{P'}\eta$ are affine maps.
\end{proof}

\noindent {\bf Remark.} The map $\eta$ is the unique involution $\psi$ in the affine group satisfying $\psi T_P = T_{P'}\psi$.\\

For the corollary, recall that the point $X=AA_3 \cd BB_3$ is the fixed point of the map $\mathcal{S}_1=T_P T_{P'}$, and $X'=AA_3' \cdot BB_3'$ is the fixed point of $\mathcal{S}_2=T_{P'}T_P$.

\begin{cor}\label{cor:3.3} Assume that the ordinary point $P$ does not lie on a median of triangle $ABC$ or on $\iota(\li)$. If $X$ is an ordinary point, then $\eta(X) = X'$ and $XX'$ is parallel to $PP'$. Thus the line joining the $P$-ceva conjugate of $Q$ and the $P'$-ceva conjugate of $Q'$ is parallel to $PP'$.
\end{cor}

\begin{proof} We have $T_{P'}T_P(\eta(X)) = T_{P'} \eta(T_{P'}(X)) = \eta T_PT_{P'}(X) = \eta(X)$, which shows that $\eta(X)$ is an ordinary fixed point of $\mathcal{S}_2=T_{P'}T_P$.  Hence, $\eta(X) = X'$. This implies the assertion, by I, Theorems 3.8 and 3.10.
\end{proof}
\bigskip

\begin{thm}\label{thm:3.5} The commutator $\cS' = T_{P'}T_PT_{P'}^{-1}T_P^{-1}$ is always a translation, in the direction $PP'$ by the distance $T_P(P')P' \cong PT_{P'}(P)$, if $P$ and $P'$ are ordinary; and by the distance $3|G_1G|$, if $P$ or $P'$ is infinite.
\end{thm}

\begin{proof}
For notational convenience write $T_1$ for $T_P$ and $T_2$ for $T_{P'}$.  We first note that
\[\cS'(T_1(P')) = T_2T_1T_2^{-1}(P') = T_2T_1(Q) = T_2(Q) = P'\]
and similarly
\[\cS'(P) = T_2T_1T_2^{-1}(Q') = T_2T_1(Q') = T_2(P).\]
From this computation and Corollary \ref{cor:3.1}, the mapping $\cS'$ fixes the line $PP'$.\smallskip

By I, Theorem 3.8 and Corollary 3.11(c) we may assume that both $\cS_1 = T_1T_2$ and $\cS_2 = T_2T_1$ are homotheties, since otherwise the assertion is trivial. If $P$ and $P'$ are ordinary points, then $\cS_2(Q) = P'$ implies that $X', P'$, and $Q$ are collinear, so part (b) of the same corollary implies that
\[\frac{X'P'}{X'Q} = \frac{T_1(X'P')}{T_1(X'Q)} = \frac{XT_1(P')}{XQ} = \frac{XP}{XQ'};\]
the last equality being a consequence of $\cS_1(Q) = T_1(P')$ and $\cS_1(Q') = P$. This equation shows that the similarity ratios of $\cS_1$ and $\cS_2$ are equal and the mapping $\cS' = \cS_2\cS_1^{-1}$ is an isometry which fixes $l_\infty$ pointwise. It follows that $\cS'$ is either a half-turn or a translation.\smallskip

First assume that $P$ does not lie on a median of $ABC$. Then $\eta(T_1(P')) = T_2(\eta(P')) = T_2(P)$, so that $T_1(P')$ and $T_2(P)$ are both inside or both outside the segment $PP'$, and $T_1(P')P' \cong PT_2(P)$ since $\eta$ preserves lengths along $PP'$. If the midpoints of segments $T_1(P')P'$ and $PT_2(P)$ are $M_1$ and $M_2$, then $\eta(M_1) = M_2$. If $M_1 = M_2$ then $M_1$ is the midpoint of $PP'$, impossible since $T_1(P') \ne P = T_1(Q')$. Thus, $M_1 \ne M_2$, and $\cS'$ fixes the line $PP'$ but is not a half-turn. Hence, $\cS'$ is a translation. This proves the assertion in the case that $P$ is not on a median.\smallskip

Now assume that $P \ne G$ lies on the median $AG$. Then $P, P', Q$, and $Q'$ are distinct points on $AG$. (This follows easily from the fact that $P, Q$ are on the opposite side of $G$ on line $AG$ from $P', Q'$ and $T_P(Q)=Q$ and $T_P(Q')=P$.  See I, Theorems 3.2 and 3.7.)  From I, Theorem 3.5 the points $A_i$ for $0 \le i \le 4$ are all the same point $A_0 = A_3$, the midpoint $EF\cd AG$ of $EF$; and similarly, the points $A'_i$ equal $A'_0 = A'_3$, the midpoint $E_3F_3\cd AG$ of $E_3F_3$. Also, $T_1(P')$ and $T_2(P)$ are on $PP' = AG$, so $T_1, T_2$ and $\cS'$ all fix the line $AG = AD = PP'$.\smallskip

If $P$ lies in the interior of triangle $ABC$, then it is easy to see that $E = T_1(B)$ and $E_3 = T_2(B)$ are on the same side of the line $AD$ as point $C$, and hence that $T_1$ and $T_2$ interchange the sides of this line. Therefore, $B_3 = T_1(E_3)$ and $B'_3 = T_2(E)$ are on the same side of line $AD = AG$, which implies that segments $B_3B'_3$ and $A_3A'_3$ do not intersect. Since
$$\cS'(A_3B_3C_3) = T_2T_1T_2^{-1}(D_3E_3F_3)=T_2T_1(ABC)=A'_3B'_3C'_3,$$
we see that $\cS'$ is not a half-turn, and is therefore a translation. On the other hand, if $P$ is exterior to $ABC$, points $E$ and $E_3$ are on opposite sides of the line $AD$, so one of $T_1$ and $T_2$ interchanges the sides of $AD$ and the other leaves both sides of $AD$ invariant. Hence $B_3 = T_1(E_3)$ and $B'_3 = T_2(E)$ are on the same side of the line $AD$ and we get the same conclusion as before.\smallskip

Finally, assume that $P' = Q$ is infinite. Using I, Theorem 3.14, we have as in the proof of Lemma \ref{lem:3.4} that $T_1(G) = G_1, T_1(G_2) = T_1T_2(G) = K^{-1}(G) = G$ and $T_1(Q) = Q$. Since $G$ is the midpoint of the segment $G_1G_2$, the fundamental theorem of projective geometry implies that $T_1$ acts as translation along the line $GG_1$. In this situation the map $\cS_1 = K^{-1}$ fixes $G$, so $X = G$, while $\cS_2 = T_2T_1$ fixes $G_2 = X'$. (See I, Theorem 3.10 and Corollary 3.11.)  Furthermore, $\cS_2(G) = T_2T_1(G) = T_1^{-1}K^{-1}(G_1) = T_1^{-1}(G_3)$, where $G_3 = K^{-1}(G_1)=T_1^{-1}(G_2)$. Hence, $\cS_2(G) = T_1^{-2}(G_2)$, and the similarity ratio of the homothety $\cS_2$ is $-2$, which is equal to the similarity ratio of $\cS_1 = K^{-1}$. Once again, $\cS' = \cS_2\cS_1^{-1} = \cS_2K$ is an isometry. It is now straightforward to verify that the map $\cS'$, which is the commutator of $T_1^{-1}$ and $K^{-1}$, is equal to the translation $T_1^{-3}$ on the line $GG_1$. For example, $\cS'(G) = \cS_2K(G) = T_1^{-2}(G_2)=T_1^{-3}(G)$, while $\cS'(G_3) = \cS_2(G_1) = T_1^{-1}K^{-1}T_1(G_1) = T_1^{-4}(G_2) = T_1^{-3}(G_3)$. Hence, the points on the line $GG_1$ experience a translation and not a half-turn. If instead, $P = Q'$ is infinite, we apply the same argument to the inverse of $\cS'$. This completes the proof.
\end{proof}

\begin{figure}
\[\includegraphics[width=4.5in]{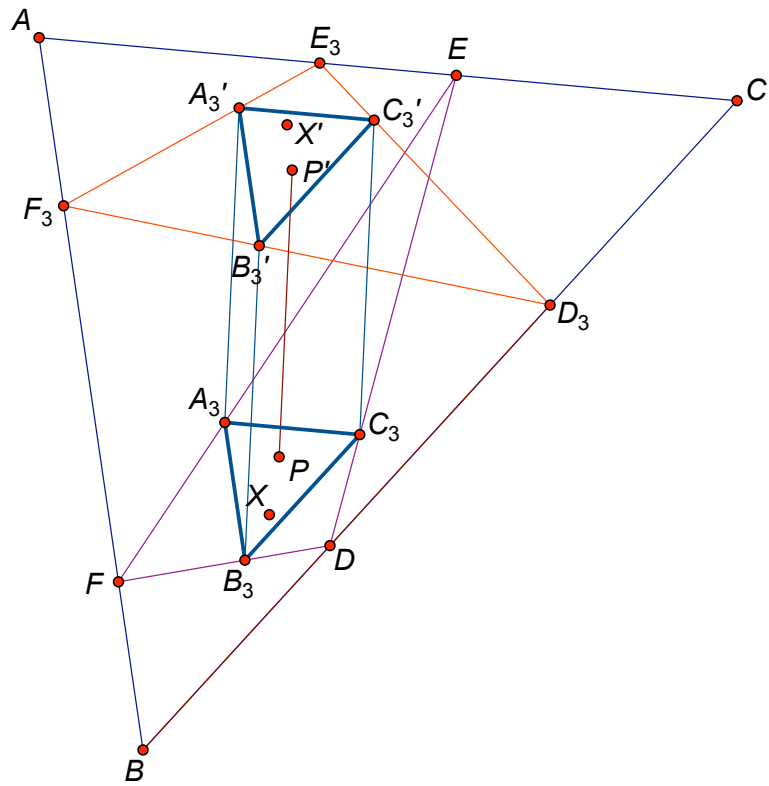}\]
\caption{$A_3B_3C_3 \cong A'_3B'_3C'_3$}
\label{fig:3.3}
\end{figure}

\begin{cor}\label{cor:3.5}
\begin{enumerate}[(a)]
\item The triangles $A_3B_3C_3$ and $A'_3B'_3C'_3$ are always congruent to each other. (See Figure \ref{fig:3.3}.)
\item If $P$ and $P'$ are ordinary, the length $|T_P(P')P|$ of the segment $T_P(P')P$ is equal to $|QQ'||XP|/|XQ'|$. 
\end{enumerate}
\end{cor}

\begin{proof} Part (a) follows from $\cS'(A_3B_3C_3) = A'_3B'_3C'_3$. Part (b) follows from $\cS_1(Q) = T_1T_2(Q) = T_1(P')$ and $\cS_1(Q') = P$, since $\mathcal{S}_1$ is a homothety with fixed point $X$.
\end{proof}

\begin{section}{The affine map $\lambda$.}
\end{section}

We now set $\la = T_{P'}T_P^{-1}$, the unique affine map taking the cevian triangle $DEF$ for $P$ to the cevian triangle $D_3E_3F_3$ for $P'$. We will show that $\la$ maps the conic $\cC_P$ to itself! \medskip

We use the fact that the diagonal triangle of any quadrangle inscribed in $\cC_P$ is self-polar [cox1, p. 73]. Since $A, B, C$ are on $\cC_P$, if $Y$ is any other point on $\cC_P$, the cevian triangle whose vertices are the traces of $Y$ on the sides of $ABC$ is a self-polar triangle. We apply this to the point $Y = P$, obtaining that $DEF$ is a self-polar triangle.\\

Associated to the vertex $D$ of the self-polar triangle $DEF$ is the harmonic homology $\eta_D$ whose center is $D$ and whose axis is $EF$. We know $P$ and $A$ are collinear with $D$ and $A_4$ (on $EF$) by the Collinearity Theorem. Furthermore, $\eta_D$ maps the conic to itself. Since $P$ and $A$ are on the conic $\cC_P$, the definition of $\eta_D$ implies that $\eta_D(A) = P$ and hence that $(AP, DA_4) = -1$. This proves
\medskip

\begin{prop}\label{prop:3.6}
The point sets $AA_4PD_1, BB_4PE_1, CC_4PF_1$ are harmonic sets.
\end{prop}

This is the key fact we use to prove the following.

\begin{thm}\label{thm:3.7} If $P$ does not lie on a median of triangle $ABC$, then the map $\la = T_{P'}T_P^{-1}$ takes the conic $\cC_P$ to itself: $\la(\cC_P) = \cC_P$.
\end{thm}

\begin{proof}
Consider the anticevian triangle of $Q$, which, by I, Corollary 3.11,  is $A'B'C'=T_{P'}^{-1}(ABC)$. The cevian triangle for $Q$ with respect to this triangle is $ABC$, so $Q$ plays the role of $P$ in the above discussion and $ABC$ plays the role of $DEF$. We have that $AQ\cd BC = D_2$, so Proposition \ref{prop:3.6} implies that $A'D_2QA$ is a harmonic set. Now we map this set by $T_P$. We have that $T_P(A'D_2QA) = T_P(A')A_2QD$, so $T_P(A')$ is the harmonic conjugate of $Q$ with respect to $D$ and $A_2$. But $A_2$ is on $EF$, so the above discussion implies that $\eta_D(Q) = T_P(A')$. Since $Q$ is on $\cC_P$, so is $T_P(A') = T_PT_{P'}^{-1}(A) = \la^{-1}(A)$. Similar arguments apply to the other vertices, so $\la^{-1}$ maps $A, B, C$ to points on $\cC_P$. Now it is easy to see that $\la^{-1}(P') = Q$ and $\la^{-1}(Q') = P$, using I, Theorems 3.2 and 3.7. Hence $\la^{-1}$ maps 5 points on $\cC_P$ to 5 other points on $\cC_P$, so we must have $\la^{-1}
 (\cC_P) = \cC_P$. This implies the assertion.
\end{proof}

If $P$ is ordinary and not on $\iota(\li)$, then the map $\la = T_{P'}T_P^{-1} = \eta T_P \eta T_P^{-1}$ is the product of $\eta$ and $\eta\la = T_P\eta T_P^{-1}$, both of which are involutions (harmonic homologies) on the extended plane, and both of which fix the conic $\cC_P$. The involution $T_P\eta T_P^{-1}$ interchanges $P$ and $Q$. Its axis of fixed points is the line $T_P(GV)$ and its center $T_P(V_\infty) = T_P(PP'\cd QQ')$ lies on $T_P(QQ') = PQ$, so that the midpoint of segment $PQ$ lies on $T_P(GV)$.\\

The map $T_P\eta T_P^{-1}$ is the map corresponding to $\eta$ for the conic $T_P(\cC_P) = T_P(ABCQ'Q) = DEFPQ$, which also lies on the points $T_P(P)$ and $T_P(P')$. Since $Q$ is the complement of $T_P(P')$ with respect to the triangle $DEF$ (I, Corollary 3.3), the conic $DEFPQ$ equals $\cC_R$ for triangle $DEF$ and the point $R = T_P(P)$. This is because the cevian triangle for $T_P(P)$ in $DEF$ is $A_1B_1C_1$, the cevian triangle for $T_P(P')$ in $DEF$ is $A_3B_3C_3$, and the points $T_P(P)$ and $T_P(P')$ are isotomic conjugates with respect to $DEF$.\\

{\bf Remark.} It is easy to show that $\eta \la \eta = \la^{-1}$. Moreover, the map $\la$ is never a projective homology. (Hint: show $T_P(V_\infty)$ is never on the line $GV$ and use [cox2, p. 56, ex. 4].)\\

The last theorem has several interesting consequences.

\begin{thm}\label{thm:3.8} If $P$ does not lie on a median of triangle $ABC$, then the 6 vertices of the anticevian triangles for $Q$ and $Q'$ (with respect to $ABC$) lie on the conic $T_P^{-1}(\cC_P) = T_{P'}^{-1}(\cC_P)$, along with the points $Q$ and $Q'$. This conic is $\cC_Q = A'B'C'QQ'$ for the anticevian triangle of $Q$, which is the same as $\cC_{Q'}$ for the anticevian triangle of $Q'$. Moreover, the vertices of the anticevian triangle of $R$ with respect to $ABC$, for any point $R$ on $T_P^{-1}(\cC_P)$, lie on $T_P^{-1}(\cC_P)$.
\end{thm}

\begin{proof} In the proof of the previous theorem we showed that $T_P(A') = \la^{-1}(A)$ lies on the conic $\cC_P$, so $A'$ lies on the conic $T_P^{-1}(\cC_P)$, as do $B'$ and $C'$, the other vertices of the anticevian triangle for $Q$. Similarly, the vertices of the anticevian triangle for $Q'$ lie on the conic $T_{P'}^{-1}(\cC_P)$. But these two conics are the same conic, since $T_{P'}T_P^{-1}(\cC_P) = \cC_P$ implies $T_P^{-1}(\cC_P) = T_{P'}^{-1}(\cC_P)$. That $Q$ and $Q'$ also lie on this conic follows from $T_P^{-1}(Q) = Q$ and $T_P^{-1}(P) = Q'$. Theorem 3.12 of Part I implies the second assertion. Since the triangle $DEF$ is self-polar for $\cC_P$, it follows that $ABC$ is a self-polar triangle for $T_P^{-1}(\cC_P)$. The last assertion of the theorem follows from the general projective fact, that for any point $R$ on a conic $\cC$, the vertices of the anticevian triangle of $R$, with respect to any self-polar triangle for the conic, also lie on $\cC$. (See [cox2, p. 91
 , ex. 3].)
\end{proof}

Note that the map $\eta$ fixes the conic $T_P^{-1}(\cC_P)$, since $\eta T_P^{-1}(\cC_P) = T_{P'}^{-1} \eta(\cC_P) = T_{P'}^{-1}(\cC_P)$.\\
\\
{\bf Remark.} If $P$ is the orthocenter of $ABC$, then $Q'$ is the circumcenter, $Q$ is the symmedian point [gr1, Thm. 7], and $P'$ is the point $X(69)$, the symmedian point of the anticomplementary triangle (see \cite{ki}).  In this case, the anticevian triangle of $Q$ is the tangential triangle \cite{wo}, so Theorem \ref{thm:3.8} implies that $T_P^{-1}(\cC_P)$ is the Stammler hyperbola [y3, p. 21].

\begin{thm}\label{thm:3.9}
Assume that the ordinary point $P$ does not lie on a median of triangle $ABC$ or on $\iota(\li)$.
\begin{enumerate}[(a)]
\item The last-named points of Theorem \ref{thm:3.1} are
\[A_0P\cd D_0Q' = \la^{-1}(A), \ B_0P\cd E_0Q' = \la^{-1}(B), \ C_0P\cd F_0Q' = \la^{-1}(C);\]
\[A'_0P'\cd D_0Q = \la(A), \ B'_0P'\cd E_0Q = \la(B), \ C'_0P'\cd F_0Q = \la(C).\]
\item Moreover, the lines $A_0P, D_0Q', DQ, A'_3P'$ are concurrent at the point $\la^{-1}(A)$, and $A'_0P', D_0Q, D_3Q', A_3P$ are concurrent at the point $\la(A)$, with similar statements holding for the other points in (a).
\end{enumerate}
\end{thm}

\begin{proof} From the proof of Theorem \ref{thm:3.7} we have that $\la^{-1}(A)$ is on the intersection of the line $DQ$ with the conic $\cC_P$ and is distinct from $Q$. The proof of Theorem \ref{thm:3.1} together with $T_{P'}(DQ) = T_{P'}(D'_3Q) = A'_3P'$ shows that $DQ\cd A'_3P'$ is also on $\cC_P$. Now $DQ\cd A'_3P'$ is not the point $Q$; otherwise $Q$ would lie on $A'_3P'$, which would imply the point $X'$ is on $A'_3P'$, since $X'$ is collinear with $Q$ and $P'$ (I, Theorem 3.8). However, the point $A$ is on $X'A'_3$ by the definition of $X'$ (I, Theorem 3.5), so $Q$ and $P'$ would be collinear with $A$. This contradicts the assumption that $P$ is not on $AG$. Hence, $DQ\cd A'_3P' = \la^{-1}(A)$, since $DQ$ intersects the conic in exactly two points.\smallskip

On the other hand, I, Corollary 3.11 and Theorems 3.12 and 2.4 give that the point $A' = T_{P'}^{-1}(A)$ is on both lines $AQ$ and $D_0Q'$, so we have that
\[\la^{-1}(A) = T_PT_{P'}^{-1}(A) = T_P(AQ\cd D_0Q') = DQ\cd A_0P.\]
Note that $DQ\cd A_0P \ne P$ since $Q$ does not lie on $PD = AP$ (see the proof of Theorem 2.1). Finally, $A_0P\cd D_0Q'$ is on $\cC_P$, by Theorem \ref{thm:3.1}, and this intersection is not $P$, since $P$ and $Q'$ are collinear with $G$ and $G$ is not on $D_0Q'$. Hence, $A_0P\cd D_0Q' = \la^{-1}(A)$. Therefore, the lines $A_0P, D_0Q', DQ$, and $A'_3P'$ meet at the point $\la^{-1}(A)$. This proves the first assertion in (b). The second assertion follows immediately upon reversing the roles of $P$ and $P'$. This gives two of the equalities in part (a), and the others follow by the same reasoning applied to the other vertices.
\end{proof}

\begin{cor} \label{cor:3.9}  Under the assumptions of the theorem, the two quadrangles $PQQ'P'$ and $A_0DD_0A_3'$ are perspective, as are quadrangles $PQQ'P'$ and $A_3D_0D_3A_0'$.  (See Part I, Figure 6, where $Y=\lambda(A)$.)
\end{cor} 
\medskip

\begin{thm} The translation $\cS' = T_{P'}T_PT_{P'}^{-1}T_P^{-1}$ of Theorem \ref{thm:3.5} maps the conic $DEFPQ$ to the conic $D_3E_3F_3P'Q'$. In other words, these two conics are {\it congruent}.
\end{thm}

\begin{proof} From Theorem \ref{thm:3.8} we have, again with $T_1=T_P$ and $T_2=T_{P'}$, that
\begin{align*}
\cS'(DEFPQ) &= T_2T_1T_2^{-1}(ABCQ'Q) = T_2T_1(T_2^{-1}(\cC_P)) \cr \\
&= T_2T_1(T_1^{-1}(\cC_P)) = T_2(\cC_P) = D_3E_3F_3P'Q'.
\end{align*}
\end{proof}

\begin{section}{The center $Z$ of $\mathcal{C}_P$.}
\end{section}

In this section we study the center $Z=Z_P$ of the conic $\mathcal{C}_P$, which we will recognize as a generalized Feuerbach point in Part III. \bigskip

\begin{thm}\label{thm:3.11} Assume that the ordinary point $P$ does not lie on a median of $ABC$ or on $\iota(\li)$. Then the center $Z$ of the conic $\cC_P$ is given by $Z = GV\cd T_P(GV)$. If $\cC_P$ is a parabola or an ellipse, $Z$ is the unique fixed point in the extended plane of the affine mapping $\la = T_{P'}T_P^{-1}$. If $\cC_P$ is a hyperbola, the infinite points on the asymptotes are also fixed, and these are the only other invariant points of $\la$.
\end{thm}

\begin{proof}
Since the map $\la$ leaves invariant the conic and the line at infinity, it fixes the pole of this line, which is $Z$.\smallskip

To prove uniqueness write the map $\la = \eta_1\eta_2$ as the product of the harmonic homologies $\eta_1 = \eta$ and $\eta_2 = \eta\la = T_P\eta T_P^{-1}$ (see the discussion following Theorem 3.2). The center of $\eta_1$ is $V_\infty$, lying on the line $PP'$, and the center of $\eta_2$ is $T_P(V_\infty)$, lying on the line $T_P(QQ') = QP$. If $R$ is any ordinary fixed point of $\la$, then $\eta_1(R) = \eta_2(R) = R'$. If $R$ is distinct from $R'$, this implies that $RR'$ is parallel to both lines $PP'$ and $QP$, which is impossible. If $R = R'$, then $R$ is fixed by both $\eta_1$ and $\eta_2$, so $R$ must be the intersection $GV\cd T_P(GV)$ of the axes of the two maps. This proves that $Z = GV\cd T_P(GV)$ if $Z$ is ordinary.\smallskip

Suppose now that the point $Z \ (=GV\cd \li)$ is an infinite point (so $\cC_P$ is a parabola) and $R$ is another infinite fixed point. Then $\eta_1(R) = \eta_2(R) = R'$ is also an infinite fixed point of $\la$, since $\la \eta(R) = \eta \la^{-1}(R) = \eta(R)$. If $R \ne Z, R'$, then $\la$ fixes three points on $\li$ and is therefore the identity on $\li$. But $\la(PQ) = Q'P'$, and $PQ\cd P'Q'$ is the ordinary point $V$ (Proposition \ref{prop:3.2}f), implying that $PQ$ cannot be parallel to $Q'P'$ and the line at infinity cannot be a range of fixed points. Thus, since $Z$ on $GV$ is fixed by $\eta_1$ we have $R = R'$, implying that $R$ is fixed by $\eta_1$ and $\eta_2$. We know that $V_\infty \ne T_P(V_\infty)$ (because $T_P(QQ') = QP$). Since $R$ is fixed by $\eta_1$ but different from $Z$, then $R = V_\infty$, which must lie on the axis $T_P(GV)$ in order to be fixed by $\eta_2$. Now $T_P(GV)$ lies on the point $G_1 = T_P(G)$ and on the midpoint $J$ of segment $PQ$ (see the comments following Theorem \ref{thm:3.7}). The point $G_1$ lies on the line $GV_\infty$ through $G$ parallel to $PP'$ (Lemma \ref{lem:3.4}), but $J$ is not on $GV_\infty$, since this line divides the segment $PQ$ in the ratio 2:1. Hence, $T_P(GV) = G_1J$ cannot possibly lie on the point $V_\infty$. Thus $R = Z$ and $Z$ is the only infinite invariant point of $\la$.\smallskip

We have already shown that the only possible ordinary fixed point of $\la$ is the point $Y = GV\cd T_P(GV)$. We now show this point is infinite when $Z$ is infinite. Assume $Y$ is an ordinary point. Then $Y$ is on the line $GV$ with $Z$, so $GV = YZ$ is an invariant line; hence its pole $V_\infty$ is also invariant. Now use the fact that $\la(P) = Q'$ and $\la(Q) = P'$. Since $V_\infty$ is invariant the line $PP'$ must be mapped to the line $QQ'$. Since the conic is also invariant, $P'$ must map to the second point of intersection of $QQ'$ with the conic, which is $Q$. Hence, $\la$ interchanges $P'$ and $Q$, implying that the line $P'Q$ is invariant; hence the infinite point on this line is invariant. This point is clearly not $V_\infty$; it is also not $Z$ because $P'Q\cd GV = G$, so that if $Z$ were on $P'Q$ then $P'Q = GV$, which is not the case since $P'$ and $Q$ are not fixed by the map $\eta$. This shows that $\la$ has three invariant points on the line at infinity, which contradicts the fact that $\la$ is not the identity on $\li$. Therefore, $\la$ has no ordinary fixed point in this case, and we conclude that $Z = GV\cd T_P(GV)$ if $Z$ is infinite. Hence, $Z = GV\cd T_P(GV)$ in all cases.\smallskip

Suppose next that $Z$ is ordinary and $\cC_P$ is a hyperbola. Let $R_1$ and $R_2$ be the infinite points on the two asymptotes. Since the point $V_\infty$ is the pole of the axis $GV$ of the map $\eta_1$, the point $I_1 = GV\cd \li$ is conjugate to $V_\infty$ on $\li$ and $GV = ZI_1$ and $ZV_\infty$ are conjugate diameters. By a theorem in Coxeter [cox1, p. 111, 8.82], these two conjugate diameters are harmonic conjugates with respect to the asymptotes. (Note that $GV$ is never an asymptote: since $GVV_\infty$ is the diagonal triangle of the quadrangle $PQQ'P'$, the pole $V_\infty$ of $GV$ never lies on $GV$.) Hence we have the harmonic relation $H(R_1R_2, I_1V_\infty)$. By the definition of $\eta_1$ this implies that $\eta_1(R_1) = R_2$. In the same way, $T_P(\Vi)$ is the center and $T_P(GV)$ is the axis of the harmonic homology $\eta_2$. Since $\eta_2$ fixes the conic, it fixes the pole of the line $T_P(GV)$ with respect to $\cC_P$; this pole is the center $T_P(\Vi)$ since $T_P(\Vi)$ is the only fixed point of $\eta_2$ off the line $T_P(GV)$. Hence we have the harmonic relation $H(R_1R_2, I_2T_P(\Vi))$, where $I_2 = T_P(GV)\cd \li$, and this implies that $\eta_2(R_1) = R_2$. Therefore, $\la(R_1) = \eta_1\eta_2(R_1) = \eta_1(R_2) = R_1$ and $\la(R_2) = R_2$. There cannot be any other fixed points since $\la$ induces a non-trivial map on the line $\li$.\smallskip

The only case left to consider is the case when $\cC_P$ is an ellipse. In this case $Z$ is an interior point of the conic and any line through $Z$ intersects the conic in two points. Suppose that the infinite point on the line $ZR_1$ is fixed by $\la$, where $R_1$ lies on the conic. As before we have $\eta_1(ZR_1) = \eta_2(ZR_1)$, so $R_2 = \eta_1(R_1)$ lies on the conic, along with $R_3 = \eta_2(R_1)$, and $Z, R_2$, and $R_3$ are collinear. The points $R_2$ and $R_3$ are distinct because $Z$ is the only ordinary fixed point of $\la$. Now line $R_1R_2$ is on the point $\Vi$ and $R_1R_3$ is on the point $T_P(\Vi)$, so $R_1R_2R_3$ is a triangle. Furthermore, $T_P(GV)$ is on the midpoint of segment $R_1R_3$. Since $Z$ is the center, it is the midpoint of the segment $R_2R_3$, and therefore the line $T_P(GV)$, which lies on $Z$ and the midpoint of $R_1R_3$, is parallel to $R_1R_2$. This implies that $T_P(GV)$ lies on $\Vi$, which we showed above is never the case. Therefore, $Z$ is the only fixed point of $\la$ in this case. This completes the proof of the theorem.
\end{proof}

\begin{figure}
\[\includegraphics[width=4.5in]{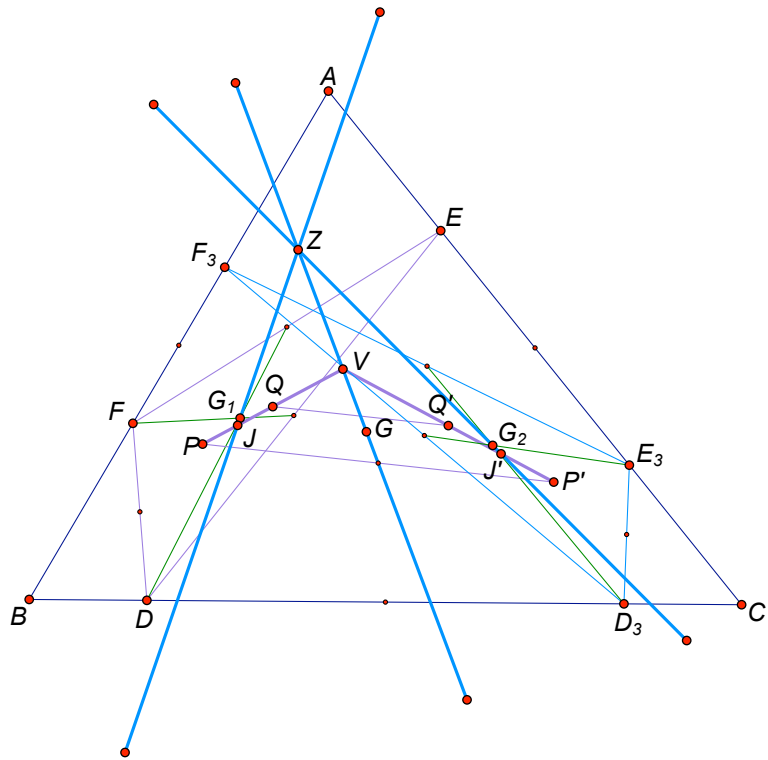}\]
\caption{The center $Z$ of the conic $\cC_P$}
\label{fig:3.4}
\end{figure}

\begin{cor}\label{cor:3.11}
\begin{enumerate}[(a)]
\item The point $Z = GV\cd G_1J$, where $G_1$ is the centroid of the cevian triangle of $P$ and $J$ is the midpoint of segment $PQ$.
\item The lines $GV, T_P(GV) = G_1J$, and $T_{P'}(GV) = G_2J' = \eta(G_1J)$ are concurrent at $Z$. (See Figure \ref{fig:3.4}.)
\end{enumerate}
\end{cor}

\begin{thm}\label{thm:3.12} Assume that the ordinary point $P$ does not lie on a median of $ABC$ but does lie on the Steiner circumellipse $\iota(\li)$. Then the point $Z$ is one-third of the distance from the point $G$ to the point $G_1 = T_P(G)$ on the line $GG_1$ and is the only ordinary fixed point of the mapping $\la$. In this case $\cC_P$ is a hyperbola and the line $l = GG_1$ is an asymptote.
\end{thm}

\begin{proof}
This will follow from the discussion in the last paragraph of the proof of Theorem \ref{thm:3.5}, according to which $T_P$ acts as translation on $l = GG_1$ and $\la = T_{P'}T_P^{-1} = T_P^{-1}K^{-1}T_P^{-1}$. First of all, it follows easily from this representation of $\la$ that $\la$ fixes the point $Z_1$ on $l$ for which $GZ_1 : GG_1 = 1:3$, and that this $Z_1$ is the only fixed point of $\la$ on this line.\smallskip

Furthermore, the mapping $T_P$ interchanges the sides of the line $GG_1$. This is because the point $D_2$ is on the line $AQ=AP'=AD_2$, which is parallel to the fixed line $l$; and $A_2 = T_P(D_2)$ is on the opposite side of $l$, because $A_2$ is a vertex of the anticomplementary triangle (I, Corollary 3.14) and the point $G$ lies on segment $AA_2$. Since $K$ interchanges the sides of the line $l$, it follows that the mapping $\la$ has the same property. Thus, $\la$ has no ordinary fixed points off of $l$, and $Z_1$ is its only ordinary fixed point.\smallskip

 On the other hand, $\la$ fixes the center $Z$ of the conic $\cC_P$. If $Z \ne Z_1$, then $Z$ is infinite and $\cC_P$ is a parabola. Now the infinite point $Q$ lies on $l$ and $\cC_P$, so $Z = Q$, and the line $l$ must intersect the parabola in a second point, which would have to be a fixed point of $\la$, since $l$ and $\cC_P$ are both invariant under $\la$. This second point on $l \cap \cC_P$ is therefore the point $Z_1$. Now all the points (other than $Z_1$) of the conic $\cC_P$ lie on one side of the tangent line $l_1$ of $Z_1$. Note that $l \neq l_1$; otherwise $Q$ on $l_1$ would imply $\li$ is on $Z_1$ by the polarity.  If $U$ and $V$ are points on $\cC_P$ on either side of $Z_1$, then they lie on opposite sides of $l$ and the intersection $UV \cdot l$ is mapped to another point on $l$ by $\la$, which must lie on the same side of $l_1$. This would say that $\la$, which fixes the tangent $l_1$, leaves invariant both sides of this line. However,
\[\la(G) = T_P^{-1}K^{-1}T_P^{-1}(G) = T_P^{-1}K^{-1}(G_2) = T_P^{-1}T_P^2(G) = G_1,\]
yet $G$ and $G_1$ are on opposite sides of $l_1$.  This contradiction shows that $Z = Z_1$, proving the first assertion. This implies that $\cC_P$ is a hyperbola and the line $l = GG_1$ is an asymptote, since $l$ is on $Z$ and cannot intersect $\cC_P$ at a point other than $Q$ for the same reason as before - such an intersection would have to be a second ordinary fixed point of $\la$. This completes the proof.
\end{proof}

To view Theorem \ref{thm:3.11} from a different perspective, consider the model of hyperbolic geometry whose points are the interior points of the conic $\cC_P$ and whose lines are the intersections of Euclidean chords on the conic with the interior of $\cC_P$. Consider the case when $\cC_P$ is an ellipse. The involutions $\eta_1 = \eta$ and $\eta_2 = T_P\eta T_P^{-1}$ fix the conic $\cC_P$ and map the interior of $\cC_P$ to itself. Furthermore, the points $R$ and $\eta(R) = R'$ lie on a line ``perpendicular" to $GV$ in this model, since $RR'$ lies on the pole of $GV$. (See Greenberg [gre, p. 309].) Note that $GV$ contains points that are interior to $\cC_P$, because $Z$ is an interior point in this case. Since $\eta$ preserves cross-ratios it is a hyperbolic isometry. Thus, $\eta$ represents reflection across the diameter $GZ$ in this model [gre, pp. 341-343]. The mapping $\eta_2 = \eta \la = T_P \eta T_P^{-1}$ is also a hyperbolic isometry fixing at least two points and thus a reflection across the diameter $T_P(GV) = G_1J = G_1Z$ [gre, p. 412]. Thus, the map $\la = \eta_1\eta_2$ represents a hyperbolic rotation about the point $Z$.\\

When $\cC_P$ is a hyperbola, then $Z$ is an exterior point. If the lines $GV$ and $T_P(GV)$ are secant lines, they both pass through the pole of the line at infinity (also a secant line) and are therefore perpendicular to $\li$ in the hyperbolic model. In this case $\la$ represents a translation along $\li$. If the axis $GV$ does not intersect the conic, then the homology $\eta_1$ interchanges the sides of the line $\li$ (in the model). Since the asymptotes separate the lines $GV$ and $Z\Vi$ (see the proof of Theorem 4.1), the point $\Vi$ is a fixed point of $\eta_1$ in the hyperbolic plane, and all lines through $\Vi$ are invariant. Hence, $\eta_1$ is a half-turn about $\Vi$. If the axis $T_P(GV)$ is a secant line, then $\la$ is an indirect isometry with no fixed points and an invariant line, and is therefore a glide [gre, p. 430]. If neither axis is a secant line, $\la$ is the product of two half-turns and is therefore a translation. When $\cC_P$ is a parabola, $Z$ is a point on the conic
  and therefore an ideal point in the hyperbolic model. In this case every line through $Z$ intersects the conic, so $GV$ and $T_P(GV)$ are secants and $\la$ is a parallel displacement [gre, p. 424].
\end{section}

\noindent Dept. of Mathematics, Carney Hall\\
Boston College\\
140 Commonwealth Ave., Chestnut Hill, Massachusetts, 02467-3806\\
{\it e-mail}: igor.minevich@bc.edu
\bigskip

\noindent Dept. of Mathematical Sciences\\
Indiana University - Purdue University at Indianapolis (IUPUI)\\
402 N. Blackford St., Indianapolis, Indiana, 46202\\
{\it e-mail}: pmorton@math.iupui.edu

\end{document}